\documentclass[12pt]{amsart}

\usepackage{a4wide}
\usepackage{psfig}
\usepackage{graphicx,eepic}
\usepackage{amsmath,amsfonts,amssymb}
\usepackage{times}
\usepackage{bm}
\usepackage{float}
\restylefloat{table}

\newcommand{\la}{\lambda}
\newcommand{\be}{\beta}

\newcommand{\e}{\varepsilon}

\newcommand{\BN}{\mathbb{N}}

\newcommand{\at}{\tilde{a}}
\newcommand{\Erdos}{Erd{\H o}s}

\newcommand{\E}{\mathcal E}

\newtheorem{lemma}{Lemma}

\newtheorem{prop}[lemma]{Proposition}
\newtheorem{thm}[lemma]{Theorem}

\theoremstyle{definition}

\newtheorem{example}[lemma]{Example}

\theoremstyle{remark}
\newtheorem{rmk}[lemma]{Remark}

\numberwithin{equation}{section} \numberwithin{table}{section}

\title[A lower bound for Garsia's entropy]
{A lower bound for Garsia's entropy \\ for certain Bernoulli convolutions}
\author{Kevin G. Hare}
\address{Department of Pure Mathematics, University of Waterloo, Waterloo, Ontario, Canada N2L 3G1. E-mail: kghare@math.uwaterloo.ca}
\thanks{Research of K. G. Hare supported, in part by NSERC of Canada.}
\thanks{Computational support provided in part by the Canadian Foundation for Innovation,
    and the Ontario Research Fund.}
\author{Nikita Sidorov}
\address{School of Mathematics, The University of Manchester, Oxford Road, Manchester M13 9PL, United Kingdom. E-mail:
sidorov@manchester.ac.uk}
\date{\today}

\begin{document}

\begin{abstract}
Let $\be\in(1,2)$ be a Pisot number and let $H_\be$ denote Garsia's entropy for the Bernoulli convolution associated with $\be$. Garsia, in 1963 showed that $H_\be<1$ for any Pisot $\be$. \nocite{Garsia63} For the Pisot numbers which satisfy $x^m=x^{m-1}+x^{m-2}+\dots+x+1$ (with $m\ge2$) Garsia's entropy has been evaluated with high precision by Alexander and Zagier for $m=2$ and later by Grabner Kirschenhofer and Tichy for $m\ge3$, and it proves to be close to 1. \nocite{AlexanderZagier91, GrabnerKirschenhoferTichy02} No other numerical values for $H_\be$ are known.

In the present paper we show that $H_\be>0.81$ for all Pisot $\be$, and improve this lower bound for certain ranges of $\beta$.  Our method is computational in nature.
\end{abstract}

\maketitle

\section{Introduction and summary}
\label{sec:intro}

Representations of real numbers in non-integer bases were introduced
by R\'enyi \cite{Re} and first studied by R\'enyi and by Parry \cite{Pa, Re}.
Let $\beta$ be a real number $>1$. A {\em $\beta$-expansion}
of the real number $x \in [0,1]$ is an infinite sequence of integers
$(a_1, a_2, a_3, \dots)$ such that $x = \sum_{n \ge 1}a_n \beta^{-n}$.
The reader is referred to Lothaire, \cite[Chapter 7]{Lo} for more on these topics.
For the purposes of this paper, we assume that $1 < \beta < 2$ and $a_i \in \{0, 1\}$.

Let $\mu_\be$ denote the {\em Bernoulli convolution} parameterized by $\be$ on $I_\be:=[0,1/(\be-1)]$, i.e.,
\[
\mu_\be(E)=\mathbb P\left\{(a_1,a_2,\dots)\in\{0,1\}^\BN : \sum_{k=1}^\infty a_k\be^{-k}\in E\right\}
\]
for any Borel set $E\subseteq I_\be$, where $\mathbb P$ is the product measure on $\{0,1\}^\BN$ with $\mathbb P(a_1=0)=\mathbb P(a_1=1)=1/2$. Since $\be<2$, it is obvious that $\text{supp}\,(\mu_\be)=I_\be$.

Bernoulli convolutions have been studied for decades (see, e.g., Peres, Schlag and Solomyak \cite{PeresSchlagSolomyak00} and Solomyak \cite{Solomyak04}), but there are still many open problems in this area. The most significant property of $\mu_\be$ is the fact that it is either absolutely continuous or purely singular (see Jessen and Wintner \cite{JessenWintner35}); \Erdos\ showed that if $\be$ is a Pisot number, then it is singular (see \cite{Erdos39}). No other $\be$ with this property have been found so far.

Recall that a number $\be>1$ is called a {\em Pisot number} if it is an algebraic integer whose Galois conjugates $h \neq \beta$ are less than 1 in modulus. Such is the golden ratio $\tau=\frac{1+\sqrt5}2$ and, more generally, the {\em multinacci numbers} $\tau_m$, the positive real root satisfying $x^m=x^{m-1}+x^{m-2}+\dots+x+1$ with $m\ge2$.
The set of Pisot numbers is typically denoted by $S$.
It has been proved by Salem that $S$ is a closed subset of $(1,\infty)$ (see \cite{Salem}). Moreover, Siegel has proved that the smallest Pisot number is the real cubic unit satisfying $x^3=x+1$ -- see \cite{Siegel}.
Amara, \cite{Amara66}, gave a complete description of the set of all limit points of the Pisot numbers in $(1,2)$.
In particular:
\begin{thm}[Amara]\label{thm:Amara}
The limit points of $S$ in $(1,2)$ are the following:
$$\varphi_1=\psi_1<\varphi_2<\psi_2<\varphi_3<\chi
<\psi_3<\varphi_4< \dots <\psi_r<\varphi_{r+1}< \dots <2$$
where
$$
\begin{cases}
\text{the minimal polynomial of\ } \varphi_r \text{\ is\ }
    \Phi_r(x) = x^{r+1}-2x^r+x-1, \\
\text{the minimal polynomial of\ } \psi_r \text{\ is\ }
    \Psi_r(x) = x^{r+1}-x^r-\dots-x-1, \\
\text{the minimal polynomial of\ } \chi \text{\ is\ }
    \mathcal{X}(x) = x^4-x^3-2x^2+1. \\
\end{cases}
$$
\end{thm}
A description of the Pisot numbers approaching these limit points
    was given by Talmoudi \cite{Talmoudi78}.
Regular Pisot numbers are defined as
    the Pisot roots of the polynomials in Table \ref{tab:regular}.
Pisot numbers that are not regular Pisot numbers are called irregular
    Pisot numbers.
For each of these limit points ($\varphi_r, \psi_r$ or $\chi$), there
    exists an $\epsilon$,
    (dependent on the limit point) such that all Pisot numbers in
    an $\epsilon$-neighbourhood of this limit point are these regular
    Pisot numbers.
The Pisot root of the defining polynomial approaches the limit point
    as $n$ tends to infinity.
It should be noted that these polynomials are not necessarily minimal, and
    may contain some cyclotomic factors.
Also, they are only guaranteed to have a Pisot number root for
    sufficiently large $n$.

\begin{table}[H]
\begin{center}
\begin{tabular}{|l|l|}
\hline
Limit Points & Defining polynomials \\
\hline
$\varphi_r$  & $\Phi_r(x) x^n\pm (x^r-x^{r-1}+1)$ \\
             & $\Phi_r(x) x^n\pm (x^r-x+1)$ \\
             & $\Phi_r(x) x^n\pm (x^r+1)(x-1)$ \\
\hline
$\psi_r$     & $\Psi_r(x) x^n\pm(x^{r+1}-1)$\\
             & $\Psi_r(x) x^n\pm(x^{r}-1)/(x-1)$ \\
\hline
$\chi$       & $\mathcal{X}(x) x^n\pm(x^3+x^2-x-1)$ \\
             & $\mathcal{X}(x) x^n\pm(x^4-x^2+1)$  \\
\hline
\end{tabular}
\end{center}
\caption{Regular Pisot numbers}
\label{tab:regular}
\end{table}

Computationally, Boyd \cite{Boyd84,Boyd85} has given an algorithm that will find all Pisot numbers
    in an interval, where, in the case of limit points, the algorithm can detect
    the limit points and compensate for them.

Garsia \cite{Garsia63} introduced a new notion associated with a Bernoulli convolution. Namely, put
\[
D_n(\be)=\left\{x\in I_\be: x=\sum_{k=1}^n a_k\be^{-k}\,\, \text{with}\,\, a_k\in\{0,1\}\right\}
\]
and for $x\in D_n(\be)$,
\begin{equation}\label{eq:pnx}
p_n(x)=\#\left\{(a_1,\dots,a_n)\in\{0,1\}^n : x=\sum_{k=1}^n a_k \be^{-k}\right\}.
\end{equation}
Finally, put
\[
H_\be^{(n)}=-\sum_{x\in D_n(\be)} \frac{p_n(x)}{2^n}\log \frac{p_n(x)}{2^n}
\]
and
\[
H_\be=\lim_{n\to\infty}\frac{H_\be^{(n)}}{n\log\be}
\]
(it was shown in \cite{Garsia63} that the limit always exists). The value $H_\be$ is called {\em Garsia's entropy}.

Obviously, if $\be$ is transcendental or algebraic but not satisfying an algebraic equation with coefficients $\{-1,0,1\}$, then all the sums $\sum_{k=1}^n a_k\be^{-k}$ are distinct, whence $p_n(x)=1$ for any $x\in D_n(\be)$, and $H_\be=\log2/\log\be>1$.

However, if $\be$ is Pisot, then it was shown in \cite{Garsia63} that $H_\be<1$ -- which means in particular that $\be$ does satisfy an equation with coefficients $\{0,\pm1\}$. Furthermore, Garsia also proved that if $H_\be<1$, then $\mu_\be$ is singular.

In 1991 Alexander and Zagier in \cite{AlexanderZagier91} managed to evaluate $H_\be$ for the golden ratio $\be=\tau$ with an astonishing accuracy. It turned out that $H_\tau$ is close to 1 -- in fact $H_\tau\approx0.9957$. Grabner, Kirschenhofer and Tichy \cite{GrabnerKirschenhoferTichy02} extended this method to the multinacci numbers; in particular, $H_{\tau_3}\approx0.9804, \ H_{\tau_4}\approx 0.9867$, etc. They also showed that $H_{\tau_m}$ is strictly increasing for $m\ge3$, and $H_{\tau_m}\to1$ as $m\to\infty$ exponentially fast.

The method suggested in \cite{AlexanderZagier91} has, however, its limitations and apparently cannot be extended to non-multinacci Pisot parameters $\be$. Consequently, no numerical value for $H_\be$ is known for any non-multinacci Pisot $\be$ -- not even a lower bound.

The main goal of this paper is to present a universal lower bound for $H_\be$ for $\beta$ a Pisot number in (1,2). We prove that $H_\be>0.81$ for all such $\be$ (Theorem~\ref{thm:main}) and improve this bound for certain ranges of $\be$ (see discussion in Remark~\ref{rmk:0.81} and Proposition~\ref{prop:<1.7}).

\section{The maximal growth exponent}

Denote by $\E_n(x;\be)$ the set of all 0-1 words of length~$n$ which may act as prefixes of $\be$-expansions of $x$. We first prove a simple characterization of this set:

\begin{lemma}\label{lem:1}
We have
\[
\E_n(x;\be)=\left\{(a_1,\dots, a_n)\in\{0,1\}^n \mid 0\le
x-\sum_{k=1}^n a_k\be^{-k} \le\frac{\be^{-n}}{\be-1} \right\}.
\]
\end{lemma}
\begin{proof} Let $(a_1,\dots,a_n)\in\E_n(x;\be)$; then the fact that there exists a $\be$-expansion of $x$ beginning with this word, implies $\sum_1^n a_k\be^{-k}\le x \le \sum_1^n a_k\be^{-k}+\frac{\be^{-n}}{\be-1}$, the second inequality following from $\sum_{n+1}^\infty a_k\be^{-k}\le \frac{\be^{-n}}{\be-1}$.

The converse follows from the fact that if $0\le y\le\frac1{\be-1}$, where $y=\be^n\bigl(x-\sum_{k=1}^n a_k\be^{-k}\bigr)$, then $y$ has a $\be$-expansion $(a_{n+1},a_{n+2},\dots)$.
\end{proof}

The following lemma will play a central role in this paper.

\begin{lemma}\label{lem:lbound}
Suppose there exists $\la\in(1,2)$ such that $\#\E_n(x;\be)=O(\la^n)$ for all $x\in I_\be$. Then
\begin{equation}\label{eq:lbound}
H_\be\ge \log_\be \frac2{\la}.
\end{equation}
\end{lemma}
\begin{proof}Let $(a_1,a_2,\dots)$ be a $\be$-expansion of $x$. Denote by $p_n(a_1,\dots,a_n)$ the number of 0-1 words $(a_1',\dots,a_n')$ such that $\sum_{k=1}^n a_k\be^{-k}=\sum_{k=1}^n a'_k\be^{-k}$. Then, as was shown by Lalley \cite[Theorems~1,2]{Lalley98},
\begin{equation}\label{eq:sqrtn}
\sqrt[n]{p_n(a_1,\dots,a_n)} \to 2\be^{-H_\be},\quad \mathbb P\text{-a.e.}\ (a_1,a_2,\dots)\in \{0,1\}^\BN.
\end{equation}
Since $p_n(a_1,\dots,a_n)\le \#\E_n(x;\be)$ for $x=\sum_{k=1}^n a_k\be^{-k}$, we have $\sqrt[n]{p_n(a_1,\dots,a_n)}\le\e_n \la$ with $\e_n\to1$, which, together with (\ref{eq:sqrtn}), implies (\ref{eq:lbound}).
\end{proof}

Define the {\em maximal growth exponent} as follows:
\begin{equation*}
\mathfrak M_\beta:=\sup_{x\in I_\be} \limsup_{n\to\infty}\sqrt[n]{\#\E_n(x;\be)}.
\end{equation*}
It follows from Lemma~\ref{lem:lbound} that
\begin{equation}\label{eq:mge}
H_\be\ge \log_\be \frac2{\mathfrak M_\beta}.
\end{equation}

Computing $\mathfrak M_\be$ explicitly for a given Pisot $\be$ looks like a difficult problem (unless $\be$ is multinacci -- see Section~\ref{sec:multi}), so our goal is to obtain good upper bounds for $\mathfrak M_\be$ for various ranges of $\be$. To do that, we will need the following simple, but useful, claim.

\begin{prop}\label{prop:upb}
If $\#\E_{n+r}(x;\be)\le R\cdot\#\E_n(x;\be)$ for all $n\ge n_0$ for
    some $n_0 \geq 1$ and some $r\ge2$, then $\mathfrak M_\be\le \sqrt[r]R$.
\end{prop}
\begin{proof}By induction,
\[
\#\E_{n_0+rk}(x;\beta)\le \#\E_{n_0}(x;\be)R^k\le 2^{n_0}R^k.
\]
Let $n \geq n_0$, and choose $k_n$ such that
$n_0+r(k_n-1)\le n<n_0+rk_n$.
Then \[ \#\E_{n}(x;\be)\le \#\E_{n_0 + r k_n}(x;\be)\le 2^{n_0}R^{k_n}.\]
The result follows from
   \[ \lim_{n\to\infty} \left(2^{n_0} R^{k_n}\right)^{1/n} =
      \lim_{n\to\infty} 2^{n_0/n} R^{k_n/n}= R^{1/r} = \sqrt[r]{R} \]
by noticing that $\frac{n_0}{n} \to 0$ and $\frac{k_n}{n} \to \frac{1}{r}$ as
     $n \to \infty$.
\end{proof}

\begin{example}
For the examples in this paper, we give only $4$~digits of precision.  In fact much higher precision was used in the computations (about $50$~digits). Let us consider a toy example showing how to apply (\ref{eq:mge}) to $\beta = \be_*\approx1.6737$, the largest root of $x^5-2x^4+x^3-x^2+x-1$ (which is a Pisot number).

Let us first determine $\#\E_{2}(x;\be_*)$, dependent upon $x$. After that we will determine $\max\limits_{x \in I_{\beta_*}} \#\E_2(x;\be_*)$.
For ease of notation, we will denote
    $m_n(\beta) = \max\limits_{x \in I_\beta} \#\E_n(x;\be)$.
Hence in this case, we are determining $m_2(\beta_*)$.
Consider the values of $x$ such that $x = \frac{a_1}{\beta} +
    \frac{a_2}{\beta^2} + \cdots$ for initial string $(a_1, a_2)$. We see that
    \[ \frac{a_1}{\beta} + \frac{a_2}{\beta^2} \leq x \leq
       \frac{a_1}{\beta} + \frac{a_2}{\beta^2} + \frac{1}{\beta^3}
       + \frac{1}{\beta^4} + \dots
       = \frac{a_1}{\beta}+ \frac{a_2}{\beta^2 }+ \frac{1/\beta^3}{1-1/\beta}.
       \]
This gives us upper and lower bounds for possible initial strings of $(a_1, a_2)$.
\begin{table}[H]
\begin{center}
\begin{tabular}{|l|l|l|}
\hline
$(a_1, a_2)$ & Lower Bound  & Upper Bound  \\
\hline
$(0,0)$  &  $0.$             & $0.5300$ \\
\hline
$(0,1)$  &  $0.3570$         & $0.8870$ \\
\hline
$(1,0)$  &  $0.5975$         & $1.1275$ \\
\hline
$(1,1)$  &  $0.9545$         & $1.4845$\\
\hline
\end{tabular}
\end{center}
\caption{Upper and lower bounds for $x$ for initial strings
         of length 2 of its $\beta$-expansion}
\label{tab:2.1}
\end{table}

We next partition possible values of $x$ in $I_\beta = [0, 1.4845]$ based on these upper and lower bound.
\begin{table}[H]
\begin{center}
\begin{tabular}{|l|l|}
\hline
Range (approx)& Possible initial string of expansion\\
\hline
$x \in (0., 0.3570)$ & $(0,0)$ \\
\hline
$x \in (0.3570, 0.5300)$ & $(0,0), (0,1)$ \\
\hline
$x \in (0.5300, 0.5975)$ & $(0,1)$ \\
\hline
$x \in (0.5975, 0.8870)$ & $(0,1), (1,0)$ \\
\hline
$x \in (0.8870, 0.9545)$ & $(1,0)$ \\
\hline
$x \in (0.9545, 1.1275) $& $(1,0), (1,1)$ \\
\hline
$x \in (1.1275, 1.4845) $ &  $(1,1)$\\
\hline
\end{tabular}
\end{center}
\caption{Initial strings $(a_1, a_2)$, depending on $x \in (0, 1.4875)$.}
\label{tab:blah}
\end{table}
This immediately shows that $m_2(\beta_*) = 2$.
Hence, by induction, $\#\E_{n+2}(x;\be_*)\le 2 \#\E_n(x;\be_*)$, whence by Proposition~\ref{prop:upb}, $\mathfrak M_{\be_*}\le\sqrt2$.
By (\ref{eq:mge}), $H_{\be_*}>\frac12\log_{\be_*} 2 \approx 0.6729$.
\end{example}

Obviously, this bound is rather crude, and in the rest of the paper we will refine this method to obtain better bounds.
One thing we need to do is show how one would use this for an entire range of $\beta$ values, instead of just for a specific value.
For instance, in the example above, we could show that $m_2(\beta) = 2$ for all
    $\beta > \tau = \frac{1+\sqrt{5}}{2}$.
In addition, we will want to show how one would do this calculation for algebraic $\beta$, where we can take advantage of the algebraic nature of $\beta$.

\section{The algorithm}

Let us consider our toy example of $\beta=\be_*$ again.
We see that for each initial string $(a_1, a_2)$, we got a lower and upper bound for possible $x = a_1 \beta^{-1} + a_2 \beta^{-2} + \cdots$. For example, for $(a_1, a_2) = (1,0)$ these were approximately $0.5975$ and $1.1275$ respectively. We then used these lower and upper bounds to partition $I_\be$ into ranges.  We next show that if the relative order of these lower and upper bounds is not changed, then the partitioning of $I_\beta$ into ranges can be done in exactly the same way.

Put $(a_1,\dots, a_{k})_L=\sum_1^k a_j\be^{-j}$ and $(a_1,\dots, a_{k})_{U}=\sum_1^k a_j\be^{-j}+\frac{\be^{-k}}{\be-1}$, i.e., $[(a_1,\dots, a_{k})_{L}, (a_1,\dots, a_{k})_{U}]$ is the interval of all possible values of $x$ whose $\be$-expansion starts with $(a_1,\dots, a_k)$. For example, $(1,0)_L = 0.5975\dots$ and $(1,0)_U = 1.1275\dots$ 
This says that if
\[
(0,0)_L < (0,1)_L < (0,0)_U < (1,0)_L < (0,1)_U < (1,1)_L < (1,0)_U < (1,1)_U,
\]
then we have
\begin{table}[H]
\begin{center}
\begin{tabular}{|l|l|}
\hline
Range & Possible initial string of $\beta$-expansion of $x$ \\
\hline
$x \in ((0,0)_L, (0,1)_L)$ & $(0,0)$ \\
\hline
$x \in ((0,1)_L, (0,0)_U)$ & $(0,0), (0,1)$ \\
\hline
$x \in ((0,0)_U, (1,0)_L)$ & $(0,1)$ \\
\hline
    \vdots           & \vdots\\
    \hline
\end{tabular}
\end{center}
\caption{Upper and lower bounds for $x$ for initial strings of length~2 of its $\beta$-expansion}
\end{table}
\noindent as the equivalent table to Table \ref{tab:2.1}.
For fixed $\beta$, these $(a_1, a_2, \dots, a_k)_L$ and $(a_1, a_2, \dots, a_k)_U$ are called
    {\em critical points for $\beta$} or simply {\em critical points}.

For each inequality, there are precise values of $\beta$ for where the inequality will hold.
For example, knowing that $\beta>1$, we get that
\[
(0,0)_U  <   (1,0)_L
\iff  \frac{\beta^{-3}}{1-\beta^{-1}}  <  \frac{1}{\beta}
\iff  \frac{1+\sqrt{5}}{2}  <   \beta
\]
So if $\beta > \tau = 1.618\dots$, then $(0,0)_U < (1,0)_L$.

This observation means that we need to determine for which
values of $\beta$ we have $(a_1, a_2)_{L/U} = (a_1',a_2')_{L/U}$. We will call these values of $\beta$ the {\em transitions points} which will affect $m_n(\beta)$.

There are some immediate observations we can make that reduces the number of equations to be checked.
\begin{itemize}
\item $(a_1, a_2)_L = (a_1', a_2')_L$ and  $(a_1, a_2)_U = (a_1', a_2')_U$
    have the same set of solutions.
\item $(a_1, a_2)_L = (a_1, a_2)_U$ has no solutions.
\item If $a_1 \leq a_1'$ and $a_2 \leq a_2'$ then none of
\begin{eqnarray*}
(a_1, a_2)_L & = & (a_1', a_2')_L  \\
(a_1, a_2)_L & = & (a_1', a_2')_U  \\
(a_1, a_2)_U & = & (a_1', a_2')_U
\end{eqnarray*}
have solutions in $I_\be$.
\end{itemize}
The first two observations were used when finding all transition points.
The last observation was made by one of the referees after all
    of the computations were completed, and hence was not used as a means
    of eliminating equations to check.

In our length 2 example again, we need to check (after elimination by the
    three observations above),
\[
\begin{array}{llll}
(0,0)_U = (0,1)_L,  &
(0,0)_U = (1,0)_L,  &
(0,0)_U = (1,1)_L  &
(0,1)_L = (1,0)_L, \\

(0,1)_L = (1,0)_U   &
(1,0)_U = (1,1)_L  &
(0,1)_U = (1,0)_L, &
(0,1)_U = (1,1)_L.
\end{array}
\]
Solving all of these equations, we see that the only transition points in $(1,2)$ for length~2 are $\sqrt{2} \approx 1.4142$ and $\tau \approx 1.6180$.

So, given that we know $m_n(\be_*) = 2$, and that we have a transition point at $\tau=1.618\dots$, we can say for all $\beta \in (\tau, 2)$ that $m_2(\beta) = 2$.
Using a similar method, we can show that for $\beta \in (\sqrt{2}, \tau)$ that $m_2(\beta) = 3$, and that for $\beta \in (1, \sqrt{2})$ that $m_2(\beta) = 4$.

It is worth noting that these results do not say what happens when $\beta = \sqrt{2}$ or $\beta = \tau$. The transition points will need to be checked separately.

There is one not so obvious, but important observation that should be made at this point. It is possible for an inequality to hold for $\beta$, where $\beta$ is in a disjoint union of intervals.

For example, we have
    \[ (0,1,1,1,1)_L < (1,0,0,0,1)_U \]
    for $\beta \in (1,\sigma) \cup (\tau, 2)$, where $\sigma^3-\sigma^2-1=0$,
    with $\sigma \approx 1.4656$.
This means that it is possible for $m_n(\beta)$ to not be an
    decreasing function with respect to $\beta$.
For example $m_5(1.81) = 3$,
    $m_5(1.85) = 4$ and  $m_5(1.88) = 3$.
This phenomenon appears to become more common for larger values of $n$.

\section{Numerical computations}

In this section we will talk about the specific computations, and how they were done.
The process started with length $n =2$, and then progressively worked on
    $n = 3, 4, 5, \dots$ up to $n=14$.
We used this process to find the global minimum for all $\beta \in (1.6, 2)$
    minus a finite set of transition points.
The code for for finding transition points, numerical lower bounds, and symbolic lower bounds can be found on the homepage of the first author \cite{HomePageWat}.

\begin{itemize}
\item For each length in order, find all solutions $\beta$ to
    \[(a_0, a_1, \dots, a_{n-1})_{L/U} = (a_0', a_1', \dots, a_{n-1}')_{L/U}\]
    subject to the conditions mentioned in the previous section.
\item For each of these solutions, check to see if the transition point is a
    Pisot number.  If so, we will have to check this transition point
    using the methods of Section \ref{sec:Symbolic}.
\item Use these transition points to partition $(1,2)$ into subintervals, upon
    which $m_n(\beta)$ is constant.
\item For the midpoint of each of these subintervals, compute
    $m_n(\beta)$,
\end{itemize}

To compute $m_n(\beta)$, we first consider all 0-1~sequences $w_1, w_2, \dots$ of length $n$. For each of these sequences, find their upper and lower bounds, say
    $\{\alpha_1, \alpha_2, \dots\}
    = \{{w_{1}}_L, {w_1}_U, {w_2}_L, \linebreak[2] {w_2}_U, \dots \}$.
Here the $\alpha_i$ are reorder such that $\alpha_i < \alpha_{i+1}$ for all $i$.
We then loop through each interval $(\alpha_i, \alpha_{i+1})$ and check how many of the $w_i$ are valid on this interval.
We keep track of the interval with the maximal set of valid $w_i$.

It should be noted that the number of times we needed to run this algorithm was rather big. At level~14, we had slightly more than $300,000$ tests where we needed to find the maximal set.

These calculations were done in Maple on 22~separate 4~CPU, 2.8~GHz machines each with 8~Gigs of RAM. These calculations were managed using the N1 Grid Engine. This cluster was capable of performing 88 simultaneous computations.

After this, we looked at all of these subintervals between transition points, and calculated the lower bounds for $H_\beta$ at the endpoints, to find a global minimum. This gives rise to the main result of the paper:

\begin{thm}
\label{thm:0.81 i}
If $\beta > 1.6$, and $\beta$ is not a transition point for $n\le14$, then $H_\beta>0.81$.
\end{thm}

\begin{rmk}
\label{rmk:0.81}
This theorem is weaker than necessary for most values of $\beta$.
For specific ranges of values of $\beta$, we actually get a number of stronger results.
\begin{itemize}
\item For most $\beta \in (1.6, 2.0)$ have $H_\beta > 0.82$, (99.9 \%),
    and a majority (51.4\%) have $H_\beta > 0.87$.
    Here ``most'' is a bit misleading.  Almost every $\beta$ has
    $H_\beta = \log2/\log\be$.
    Of those that do not, there is no result that shows they
    should be evenly distributed, (and they most likely are not).
    So by ``most'' we mean that for some finite collection of intervals, that
    make up $99.9\%$ of $(1.6, 2.0)$ that {\em all} $\beta$ in this
    finite collection of intervals have $H_\beta > 0.82$.
\item The minimum occurs near $\tau_3 \approx 1.8392$,
    (See Figure~\ref{fig:1.83,1.85}).
\item For $\beta \in (1.6, 1.7)$ we have $H_\beta > 0.87$
    (Figure~\ref{fig:1.6,1.7}), and for $\beta$ near $2.0$ we have
    $H_\beta> 0.9$ (Figure~\ref{fig:1.98,2.0}).
\end{itemize}
\end{rmk}

\begin{figure}[H]
\[ \psfig{file=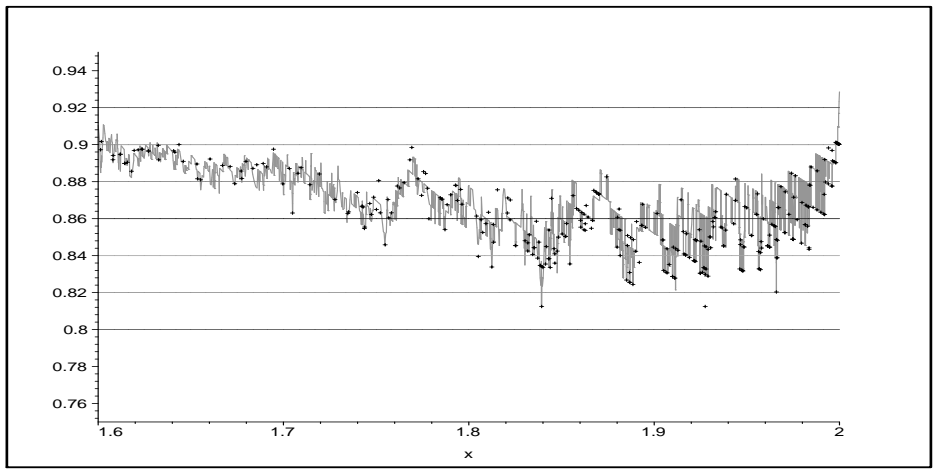,height=175pt,width=350pt,angle=270} \]
\caption{Lower bound for $H_\beta$, for Pisot $\beta \in (1.6, 2.0)$ and Pisot
    Transition points}
\label{fig:1.6,2.0}
\end{figure}
\begin{figure}[H]
\[ \psfig{file=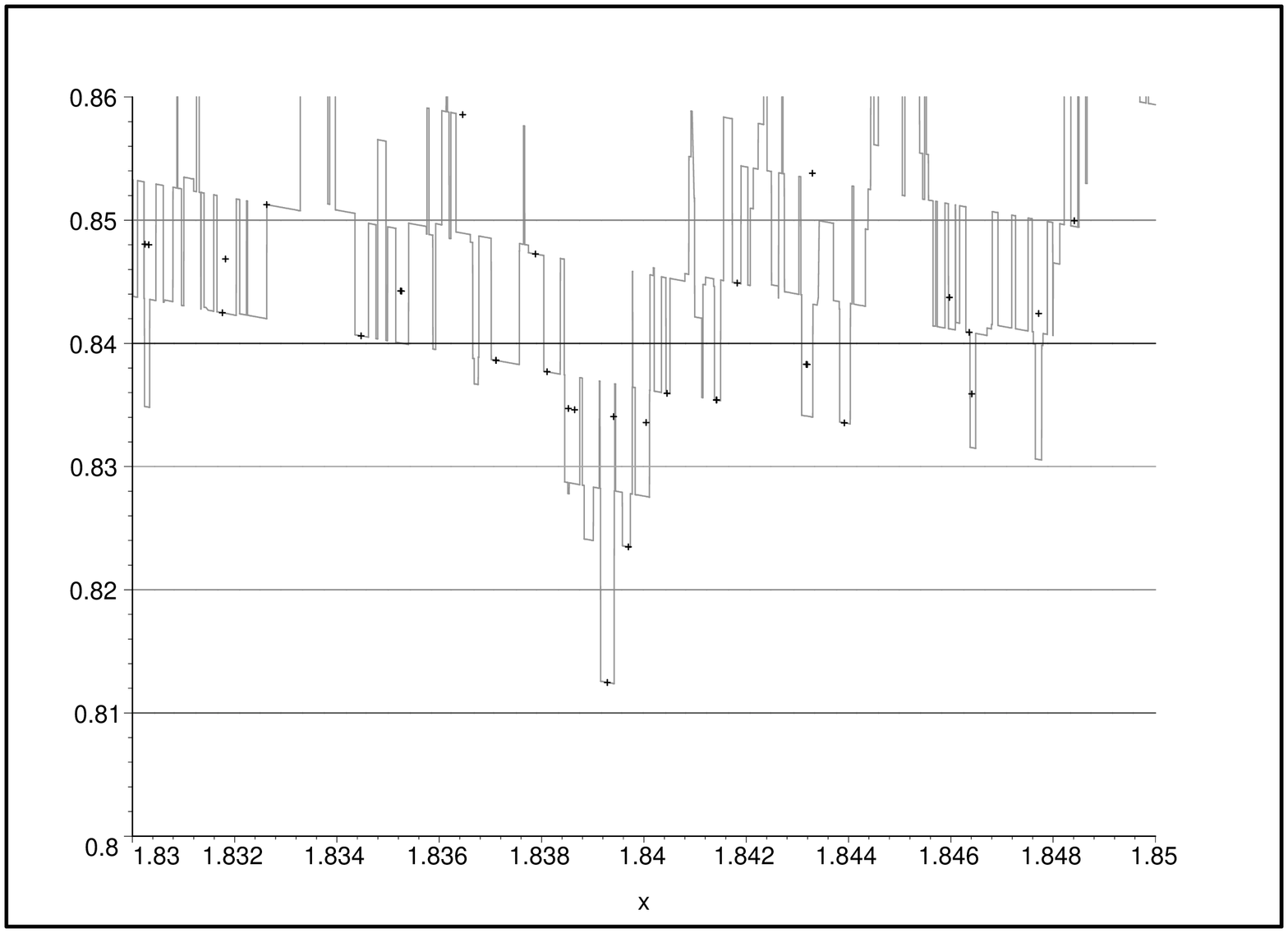,height=175pt,width=350pt,angle=270} \]
\caption{Lower bound for $H_\beta$, for Pisot $\beta \in (1.83, 1.85)$ and Pisot
    Transition points}
\label{fig:1.83,1.85}
\end{figure}
\begin{figure}[H]
\[ \psfig{file=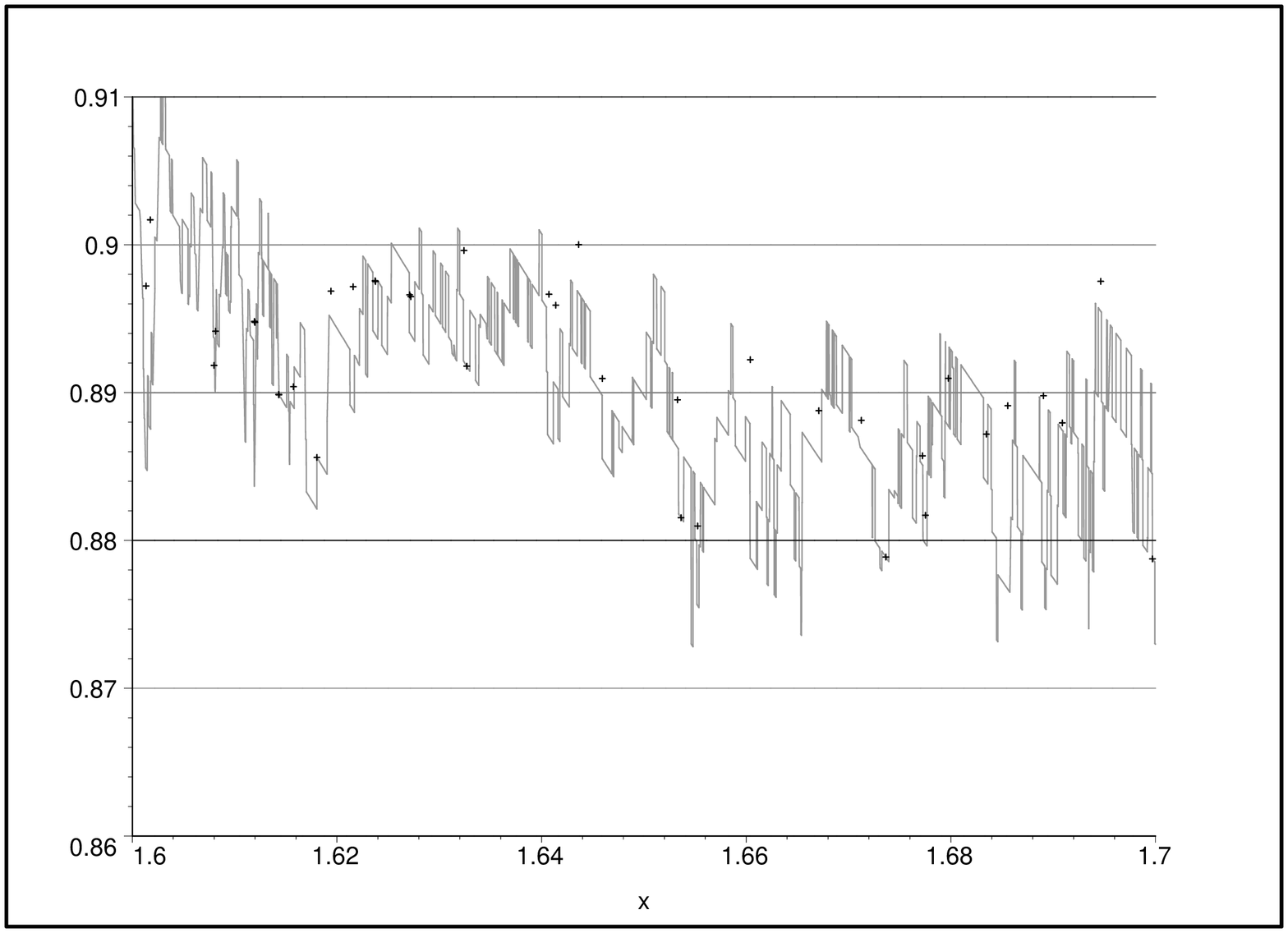,height=175pt,width=350pt,angle=270} \]
\caption{Lower bound for $H_\beta$, for Pisot $\beta \in (1.6, 1.7)$ and Pisot
    Transition points}
\label{fig:1.6,1.7}
\end{figure}
\begin{figure}[H]
\[ \psfig{file=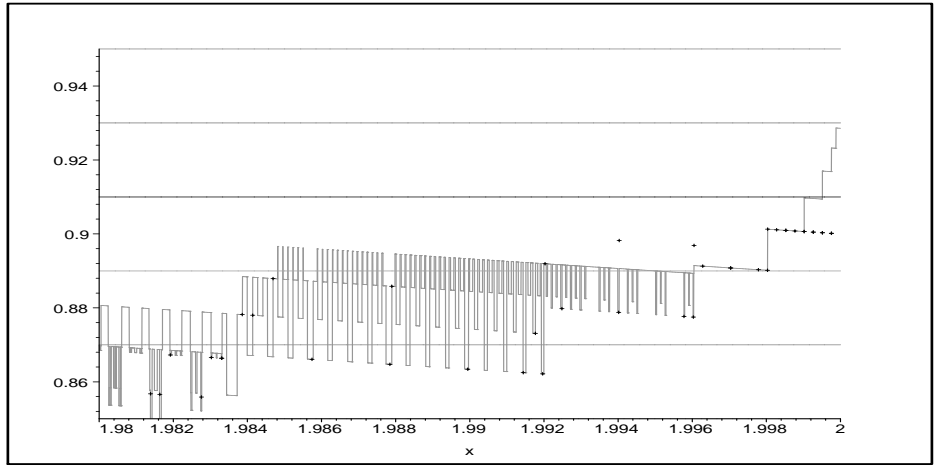,height=175pt,width=350pt,angle=270} \]
\caption{Lower bound for $H_\beta$, for Pisot $\beta \in (1.98, 2.0)$ and Pisot
    Transition points}
\label{fig:1.98,2.0}
\end{figure}

\section{Calculations for symbolic $\beta$}
\label{sec:Symbolic}

In the previous section, we showed for all but a finite number of Pisot numbers $\beta$ in $(1.6, 2)$ that $H_\beta > 0.81$.  To extend the result to all such $\beta$ in $(1,2)$, there are still some of Pisot numbers that will need to be checked individually.

These include the finite set of Pisot numbers less that 1.6
(of which there are 12), and the finite set of Pisot numbers that are also transition points (of which there are 427).
In particular, we get:

\begin{thm}
\label{thm:0.81 ii}
For all Pisot numbers $\beta < 1.6$ and all Pisot transition points (for $n \leq 14$), we have $H_\beta > 0.81$.
\end{thm}

Combined, this theorem and Theorem~\ref{thm:0.81 i} yield

\begin{thm}\label{thm:main}
For any Pisot $\be$ we have $H_\be>0.81$.
\end{thm}

\begin{table}[H]
\begin{center}
\begin{tabular}{|l|l|l|l|}
\hline
Minimal polynomial of $\beta$ & Pisot number & Length & Lower Bound for $H_\be$ \\
\hline
$x^3-x-1 $&$ 1.3247 $&$ 17 $&$ .88219$ \\ \hline
$x^4-x^3-1 $&$ 1.3803 $&$ 16 $&$ .87618$ \\ \hline
$x^5-x^4-x^3+x^2-1 $&$ 1.4433 $&$ 15 $&$ .89257$ \\ \hline
$x^3-x^2-1 $&$ 1.4656 $&$ 15 $&$ .88755$ \\ \hline
$x^6-x^5-x^4+x^2-1$&$ 1.5016 $&$ 14 $&$ .90307$ \\ \hline
$x^5-x^3-x^2-x-1 $&$ 1.5342 $&$ 15 $&$ .89315$ \\ \hline
$x^7-x^6-x^5+x^2-1 $&$ 1.5452 $&$ 13 $&$ .90132$ \\ \hline
$x^6-2 x^5+x^4-x^2+x-1 $&$ 1.5618 $&$ 15 $&$ .90719$ \\ \hline
$x^5-x^4-x^2-1 $&$ 1.5701 $&$ 15 $&$ .88883$ \\ \hline
$x^8-x^7-x^6+x^2-1 $&$ 1.5737 $&$ 14 $&$ .90326$ \\ \hline
$x^7-x^5-x^4-x^3-x^2-x-1 $&$ 1.5900 $&$ 15 $&$ .89908$ \\ \hline
$x^9-x^8-x^7+x^2-1 $&$ 1.5912 $&$ 14 $&$ .90023$ \\ \hline
\end{tabular}
\end{center}
\caption{Lower bounds for Garsia's entropy for all Pisot numbers $< 1.6$}
\label{tab:0.81}
\end{table}

As a corollary, we obtain a result on small Pisot numbers:

\begin{prop}
\label{prop:<1.7}
All Pisot $\be<1.7$ have Garsia entropy $H_\be>0.87$.
\end{prop}

There are actually a lot of advantages to doing a symbolic check as compared to the numerical techniques of the previous section.
Some of these include not requiring high precision arithmetic
    and the combining of equivalent strings, both of
    which has speed and memory advantages.
These are described in the example below.

To illustrate the (computer-assisted) proof of Theorem~\ref{thm:0.81 ii}, consider as an example $\beta = \tau$ the golden ratio. As before, we wish to find the
    \[ \frac{a_1}{\tau} + \frac{a_2}{\tau^2} \leq x \leq
        \frac{a_1}{\tau}+ \frac{a_2}{\tau^2 }+ \frac{1/\tau^3}{1-1/\tau}. \]
But now we can find exact symbolic values for these ranges.
In particular, we notice that $\frac{1/\tau^3}{1-1/\tau} = \tau - 1$. Secondly, as $\frac{1}{\tau} = \tau-1$ and $\frac{1}{\tau^2} = 2-\tau$ we get

\begin{table}[H]
\begin{center}
\begin{tabular}{|l|l|l|}
\hline
$(a_1, a_2)$ & Lower Bound         & Upper Bound \\
\hline
$(0,0)$  &  $0                   $ & $\tau-1 \approx 0.618 $\\
\hline
$(0,1)$  &  $2-\tau \approx 0.382$ & $1 $\\
\hline
$(1,0)$  &  $\tau-1\approx 0.618 $ & $2\tau-2 \approx 1.236$ \\
\hline
$(1,1)$  &  $1                   $ & $\tau \approx 1.618 $\\
\hline
\end{tabular}
\end{center}
\caption{Upper and lower bounds for initial strings of length 2 for $x = a_1 \tau^{-1} + a_2 \tau^{-2} + \dots$}
\end{table}

So in particular, it {\em is} possible for $x$ to start with both $(0,0)$ and $(1,0)$. But if this is the case then $x = (0,0,1,1,1,\dots)  = (1,0,0,0,\dots) = \tau -1$. So it is not possible for $x$ to have an infinite number of expansions starting with $(0,0)$ and an infinite number of expansions starting with $(1,0)$.
Similar arguments can be used for the other critical point, $x=1$.

So we can discard the critical points and subdivide the possible values of $x$ into the following ranges:

\begin{table}[H]
\begin{center}
\begin{tabular}{|l|l|}
\hline
Range & Possible initial string of the $\tau$-expansion \\
\hline
$x \in (0, 2-\tau)$ & $(0,0)$ \\
\hline
$x \in (2-\tau, \tau-1) $ & $(0,0), (0,1)$ \\
\hline
$x \in (\tau-1, 1) $ & $(0,1), (1,0)$ \\
\hline
$x \in (1, 2 \tau - 2) $ & $(1,0), (1,1)$ \\
\hline
$x \in (2 \tau - 2, \tau) $ & $(1,1)$ \\
\hline
\end{tabular}
\end{center}
\caption{Initial string of $\tau$-expansion of $x$, depending on $x$.}
\end{table}

This immediately shows that $m_2(\tau) = 2$.
Hence, by induction, $\#\E_{n+2}(x;\tau)\le 2 \#\E_n(x;\tau)$, whence $\mathfrak M_{\tau}\le\sqrt2$. By (\ref{eq:mge}), $H_{\tau}>\frac12\log_{\tau} 2 = 0.7202100$.

The main advantage of this method comes when we have longer strings. In particular, it is easy to see that $(1,0,0) = (0,1,1)$.
This allows us to compress information.

\begin{table}[H]
\begin{center}
\begin{tabular}{|l|l|l|}
\hline
$a_1 a_2 a_3$    & Lower Bound         & Upper Bound \\
\hline
$(0,0,0)$          & 0                     & $5-3 \tau \approx 0.1459$  \\
\hline
$(0,0,1)$          & $2\tau-3\approx 0.2361$& $2- \tau \approx 0.3820$ \\
\hline
$(0,1,0)       $   & $2-\tau\approx 0.3820  $& $4-2\tau \approx 0.7639 $\\
\hline
$(0,1,1) = (1,0,0) $   & $\tau-1\approx 0.6180  $& $1$ \\
\hline
$(1,0,1)       $   & $3\tau-4\approx 0.8541 $& $2\tau-2 \approx 1.2361 $\\
\hline
$(1,1,0)       $   & $1$                     & $3 - \tau \approx 1.3820$\\
\hline
$(1,1,1)       $   & $2\tau-2\approx 1.2361$& $\tau \approx 1.6180$\\
\hline
\end{tabular}
\end{center}
\caption{Upper and lower bounds for initial string of length $3$ for $x =
    a_1 \tau^{-1} + a_2 \tau^{-2} + a_3 \tau^{-3} + \dots $}
\end{table}

This gives that for $x \in (\tau-1, 4-2\tau)$ we have the initial string of $(0,1,0)$, $(0,1,1)$, $(1,0,0)$, and if $x \in (3\tau-4, 1)$ we have the initial string of $(1,0,1), (0,1,1), (1,0,0)$.

Our implementation does not maintain a separate entry for $(0,1,1)$ and $(1,0,0)$, as they are
    equivalent.
Instead, the algorithm stores only one of these two strings, and indicates that this
    has weight 2.
For the general Pisot $\beta$, this is checked by noticing that $(a_1, a_2, \dots, a_n)$ is equivalent to the same word as $(b_1, b_2, \dots, b_n)$ if and only if $a_n x^{n-1} + \dots + a_1 \equiv b_n x^{n-1} + b_{n-1} x^{n-2} + \dots + b_1 \equiv c_{d-1} x^{d-1} + \dots + c_d \pmod{p(x)}$ for some $c_i$, with $p(x)$ the minimal polynomial for $\beta$, of degree $d$.  Given the large amount of overlapping that we see for large lengths, this will have major cost savings, both in memory and time.

\section{The maximal growth exponent for the multinacci family and discussion}
\label{sec:multi}

In this section we will compute the maximal growth exponent for the multinacci family and compare our lower bound (\ref{eq:mge}) with the actual values.

Let, as above, $\tau_m$ denote the largest root of $x^m-x^{m-1}-\dots-x-1$ (hence $\tau=\tau_2$). Define the {\em local dimension} of the Bernoulli convolution $\mu_\be$ as follows:
\[
d_\be(x)=\lim_{h\to0}\frac{\log\mu_\be(x-h,x+h)}{\log h}
\]
(if the limit exists). As was shown in Lalley \cite{Lalley98}, $d_\be(x)\equiv H_\be$ for $\mu_\be$-a.e. $x\in I_\be$ for any Pisot $\be$.

Notice that it is well known that the limit in question exists if it does so along the subsequence $h=c\be^{-n}$ for any fixed $c>0$ (see, e.g., Feng \cite{Feng05}). We choose $c=(\be-1)^{-1}$, so
\begin{equation}\label{eq:dbetan}
d_\be(x)=-\lim_{n\to\infty}\frac1n\log_\be\mu_\be\left(x-
\frac{\be^{-n}}{\be-1},x+
\frac{\be^{-n}}{\be-1}\right).
\end{equation}
Let $\be=\tau_m$ for some $m\ge2$.

\begin{lemma}
Suppose $\be$ is multinacci, and put
\[
\e_\be(x)=\lim_{n\to\infty}\sqrt[n]{\#\E_n(x;\be)}.
\]
This limit exists if and only if $d_\be(x)$ exists, and, in this case,
\begin{equation*}
d_\be(x)=\log_\be\frac2{\e_\be(x)}.
\end{equation*}
\end{lemma}
\begin{proof}
Let $x = \sum_{k=1}^\infty a_k \beta^{-k}$ and consider $(a_1, \dots, a_n)$,
    the first $n$ terms of this sequence.
We see that
\begin{eqnarray*}
(a_1, \dots, a_n)_L
     & = & \sum_{k=1}^n a_k \beta^{-k} \\
     & \geq & \sum_{k=1}^n a_k \beta^{-k} + \sum_{k=n+1}^\infty
         (a_k-1)\beta^{-k}  \\
     & = &  \sum_{k=1}^\infty a_k \beta^{-k} - \sum_{k=n+1}^\infty \beta^{-k} \\
     & = & x - \frac{\be^n}{\beta-1}
\end{eqnarray*}
and
\begin{eqnarray*}
(a_1, \dots, a_n)_U
   & = &  \sum_{k=1}^n a_k \beta^{-k} + \sum_{k=n+1}^\infty \beta^{-k} \\
   & \leq &  \sum_{k=1}^\infty a_k \beta^{-k} + \sum_{k=n+1}^\infty
             \beta^{-k} \\
   & = &  x + \frac{\be^n}{\beta-1}
\end{eqnarray*}
Further, this true, regardless of which representation $(a_1, a_2, \dots)$
    of $x$ that we take.
Hence, if $(a_1, \dots, a_n) \in \E_n(x, \beta)$, then for {\em all}
    $a_{n+1}', a_{n+2}', \dots \in \{0, 1\}$ we have
\begin{eqnarray*}
\sum_{k=1}^{n} a_k \be^{-k} + \sum_{k=n+1}^{\infty} a_{k}' \be^{-k}
      & \in&  ((a_1, a_2, \dots, a_n)_L, (a_1, a_2, \dots, a_n)_U)  \\
      & \subseteq &
      \left(x - \frac{\be^{-n}}{\beta-1}, x+\frac{\be^{-n}}{\beta-1}
              \right)
\end{eqnarray*}
This in turn implies that
\begin{equation}\label{eq:mubeta}
\mu_\be \left(x-\frac{\be^{-n}}{\be-1},x+\frac{\be^{-n}}{\be-1}\right) \ge 2^{-n}\# \E_n(x;\be).
\end{equation}

Now put
\[
\widetilde \E_n(x;\be)=\left\{(\at_1,\dots, \at_n)\in\{0,1\}^n \mid -\frac{\be^{-n}}{\be-1}\le
x-\sum_{k=1}^n \at_k\be^{-k} \le\frac{2\be^{-n}}{\be-1} \right\}.
\]
Our next goal is to prove the inequality
\begin{equation}\label{eq:mubeta2}
\mu_\be \left(x-\frac{\be^{-n}}{\be-1},x+\frac{\be^{-n}}{\be-1}\right) \le 2^{-n}\# \widetilde\E_n(x;\be).
\end{equation}

Let $y \in (x - \frac{\be^{-n}}{\beta-1}, x+ \frac{\be^{-n}}{\beta-1})$
    have an expansion $y = \sum_{k=1}^\infty \at_k \be^{-k}$.
It suffices to show that $(\at_1, \dots, \at_n) \in \widetilde\E_n(x;\be)$.

By noticing that $-\frac{\be^{-n}}{\beta-1} \leq x - y  \leq
    \frac{\be^{-n}}{\beta-1}$ we get first that
\[
-\frac{\be^{-n}}{\beta-1}  \leq  x - y
  \leq  x - \sum_{k=1}^\infty \at_k \be^{-k}
  \leq  x - \sum_{k=1}^n \at_k \be^{-k}
\]
and further than
\[
\begin{array}{crcl}
& x - y & \leq & \frac{\be^{-n}}{\beta-1} \\
\implies & x - \sum_{k=1}^\infty \at_k \be^{-k}
          & \leq & \frac{\be^{-n}}{\beta-1} \\
\implies & x - \sum_{k=1}^n \at_k \be^{-k}
   - \sum_{k=n+1}^\infty \at_k \be^{-k} & \leq & \frac{\be^{-n}}{\beta-1} \\
\implies & x - \sum_{k=1}^n \at_k \be^{-k}
   & \leq & \sum_{k=n+1}^\infty \at_k \be^{-k} + \frac{\be^{-n}}{\beta-1} \\
\implies & x - \sum_{k=1}^n \at_k \be^{-k}
   & \leq & \sum_{k=n+1}^\infty \be^{-k} + \frac{\be^{-n}}{\beta-1} \\
\implies & x - \sum_{k=1}^n \at_k \be^{-k}
   & \leq & 2 \frac{\be^{-n}}{\beta-1}
\end{array}
\]
Hence $(\at_1, \dots, \at_n) \in \widetilde E_n(x; \be)$ as required.

Combining (\ref{eq:mubeta}) and (\ref{eq:mubeta2}), we obtain
\[
2^{-n}\# \E_n(x;\be)\le \mu_\be \left(x-\frac{\be^{-n}}{\be-1},x+\frac{\be^{-n}}{\be-1}\right) \le 2^{-n}\# \widetilde\E_n(x;\be),
\]
whence
\begin{equation}\label{eq:squeeze}
\begin{aligned}
\log_\be 2-\frac1n\log_\be\# \widetilde\E_n(x;\be) &\le -\frac1n\log_\be \mu_\be\left(x-\frac{\be^{-n}}{\be-1},x+\frac{\be^{-n}}{\be-1} \right)\\
&\le \log_\be 2-\frac1n\log_\be\# \E_n(x;\be).
\end{aligned}
\end{equation}

Notice that (\ref{eq:squeeze}) in fact holds for any $\be$. Now we use the fact that $\be$ is multinacci. It follows from Feng, \cite[Lemma~2.11]{Feng05} that for a multinacci $\be$ one has $\sqrt[n]{p_n(x)}\sim \sqrt[n]{p_n(x')}$ provided $|x-x'|\le C\be^{-n}$ for any fixed $C>0$ and any $x,x'\in D_n(\be)$ which are not endpoints of $I_\be$. (Here $p_n(x)$ is given by (\ref{eq:pnx}).)

Observe that
\begin{align*}
\#\E_n(x;\be) &= \sum_{\substack{y\in D_n(\be):\\ 0\le y-x\le \frac{\be^{-n}}{\be-1}}} p_n(y),
 \\
\#\widetilde\E_n(x;\be) &= \sum_{\substack{y\in D_n(\be):\\ -\frac{\be^{-n}}{\be-1}\le y-x\le \frac{2\be^{-n}}{\be-1}}} p_n(y).
\end{align*}
In view of the Garsia separation lemma (see \cite[Lemma~1.51]{Garsia62}), each sum runs along a finite set whose cardinality is bounded by some constant (depending on $\be$) for all $n$.

Hence $\sqrt[n]{\#\E_n(x;\be)}\sim \sqrt[n]{\#\widetilde\E_n(x;\be)}$ for all $x\in(0,\frac1{\be-1})$, and (\ref{eq:squeeze}) together with (\ref{eq:dbetan}) yield the claim of the lemma.

\end{proof}

Consequently, for a multinacci $\be$,
\begin{equation}\label{eq:infdbeta}
\inf_{x\in I_\be^*} d_\be(x)=\log_\be\frac2{\mathfrak M_\be},
\end{equation}
where $I_\be^*=\bigl\{x\in(0,\frac1{\be-1}):d_\be(x)\,\text{exists} \bigr\}$.
In \cite[Theorem~1.5]{Feng05} Feng showed that
\[
\inf_{x\in I_{\tau_m}^*} d_{\tau_m}(x)=
\begin{cases}\log_\tau 2-\frac12,&m=2\\
\frac{m}{m+1}\log_{\tau_m}2,&m\ge3.
\end{cases}
\]
This immediately gives us the explicit formulae for the maximal growth exponent for the multinacci family, namely,
\[
\mathfrak M_{\tau_m}=\begin{cases}\sqrt\tau,&m=2\\
2^{\frac1{m+1}}, &m\ge3.
\end{cases}
\]
In fact, one can easily obtain the values $x$ at which $\mathfrak M_\be$ is attained. More precisely, for $\be=\tau$ the maximum growth is attained at $x$ with the $\be$-expansion $(1000)^\infty$, i.e., at $x=(5+\sqrt5)/10$.\footnote{This was essentially proved by Pushkarev \cite{Pushkarev95}, via multizigzag lattices techniques.}

For $m\ge3$ the maximal growth point is $x$ with the $\be$-expansion $(10^m)^\infty$. These claims can be easily verified via the matrix representation for $p_n(x)$ given in \cite{Feng05}, and we leave it as an exercise for the interested reader. (Recall that the growth exponent for $p_n(x)$ is the same as for $\#\E_n(x;\be)$ for the multinacci case.)

Finally, since we know the exact values of the maximal growth exponent for this family, we can assess how far our estimate (that is, the smallest value of the local dimension) is from the actual value of $H_\be$ (which is the average value of $d_\be(x)$ for $\mu_\be$-a.e. $x$). Here is the comparison table:

\begin{table}[H]
\begin{center}
\begin{tabular}{|l|l|l|}
\hline
$m$    & $\log_{\tau_m} \frac2{\mathfrak M_{\tau_m}}$   & $H_{\tau_m}$ \\
\hline
          2& $0.9404$& $0.9957$\\
          \hline
3    & $0.8531$& $0.9804$ \\
\hline
          4& $0.8450$& $0.9869$\\
          \hline
5 & $0.8545$ & $0.9926$\\
\hline
\end{tabular}
\end{center}
\caption{Lower bounds and the actual values for $H_{\tau_m}$}
\end{table}

We see that for $m\ge3$ our bounds are far below $H_\be$; moreover, our method cannot in principle produce a uniform lower bound for all $\be$ better than $0.845$. However, as a first approximation it still looks pretty good.

\begin{rmk} We believe (\ref{eq:infdbeta}) holds for all Pisot $\be\in(1,2)$. If this were the case, then (\ref{eq:mge}) would effectively yield a lower bound for the infimum of the local dimension of $\mu_\be$. This may prove useful, as, similarly to the entropy, no lower bound for $d_\be$ is known for the non-multinacci $\be$. Furthermore, if one could compute the exact value of $\mathfrak M_\be$, this would yield the exact value of $\inf_{x\in I_\be^*} d_\be(x)$.
\end{rmk}



\section{Acknowledgements and additional comments}

The authors would like to thank the two referees for many useful suggestions.
In addition, we would like to communicate a question asked to us from one of the
    referees, that the authors feel would make an interesting question for
    possible future research.
\begin{quote}
    In Section 6 - besides the multinacci, could you say something on
    $\beta = (a+\sqrt{a^2+4})$, with an integer $a \ge 2$?
    (Maybe using results from Komatsu \cite{Komatsu02}.)
    Or, more generally, on numbers $\beta$ that are root of a
    polynomial $X^n - a_{n-1}
      X^{n-1} - \dots - a_1 X - a_0$, where $a_{n-1} \ge \dots \ge a_1 \ge 1$?
\end{quote}

We would also like to mention the recent paper by Feng and the second author \cite{FengSidorov}, in which the average growth exponent for $\beta$-expansions is studied for the Pisot parameters~$\beta$.


\def\lfhook#1{\setbox0=\hbox{#1}{\ooalign{\hidewidth
  \lower1.5ex\hbox{'}\hidewidth\crcr\unhbox0}}}

\end{document}